\documentclass[11pt]{article}
\usepackage{amsfonts,amsmath,amssymb,amsthm, amscd, subcaption}
\usepackage{graphicx}
\numberwithin{equation}{section}

\newtheorem{theorem}{Theorem}[section]
\newtheorem{prop}[theorem]{Proposition}
\newtheorem{lemma}[theorem]{Lemma}
\newtheorem{cor}[theorem]{Corollary}

\theoremstyle{definition}
\newtheorem{definition}[theorem]{Definition}
\newtheorem{example}[theorem]{Example}
\newtheorem{remark}[theorem]{Remark}
\newtheorem{question}[theorem]{Question}

\newcommand{\D}{\Delta}

\def\<{{\langle}}
\def\>{{\rangle}}
\def\G{{\Gamma}}

\def\b{{\beta}}
\def\g{{\gamma}}
\def\d{{\delta}}
\def\Z{\mathbb Z}

\def\R{\mathbb R}
\def\T{{\mathbb T}}
\def\S{{\mathbb S}}

\def\D{\cal D}

\def\b{\beta}

\def\si{\sigma}
\def\t{\tau}
\def\k{{\kappa}}

\def\Gr{\Bbb G}
\def\L{\cal L}
\def\La{\Lambda}

\def\Rd{{\cal R}_d}
\def\s{{\bf s}}
\def\e{\epsilon}

\def\De{D}
\def\o{\overline}

\def\ni{\noindent} 

\begin{document}

\title{Graph Complexity and Mahler Measure}

\author{Daniel S. Silver 
\and Susan G. Williams\thanks {The authors are grateful for the support of the Simons Foundation.} }

\maketitle 


\begin{abstract} The \emph{(torsion) complexity} of a finite edge-weighted graph is defined to be the order of the torsion subgroup of the abelian group presented by its Laplacian matrix. When $G$ is $d$-periodic (i.e., $G$ has a free $\Z^d$-action by graph automorphisms with finite quotient) the Mahler measure of its Laplacian determinant polynomial is the growth rate of the complexity of  finite quotients of $G$. Lehmer's question, an open question about the roots of monic integral polynomials, is equivalent to a question about the complexity growth of edge-weighted 1-periodic graphs.   \bigskip

\end{abstract}

\section{Introduction.} \label{Intro} 
The complexity of a finite graph is often defined as the number of its spanning trees.  Here we consider graphs with integer edge weights and take a different approach, defining  complexity to be the order of the torsion subgroup of the abelian group presented by the Laplacian matrix of the graph. When $G$ is connected and all edge weights are $1$, the complexity as we define it is the number of spanning trees of the graph. However, for general  edge-weighted graphs, the two notions of complexity are different. 

Our main objects of study are \emph{$d$-periodic graphs}, infinite graphs $G$ on which $\Z^d, \ d \ge 1$, acts freely by graph automorphisms such that integer edge weights are preserved and the quotient graph $\o G$ is finite. Our motivation comes from two sources:  knot theory, where finite graphs with edge weights $\pm 1$ correspond to diagrams of knots and links, and Lehmer's question concerning the roots of integral polynomials. 

For $d$-periodic graphs a Laplacian operator is defined (see \cite{Ke11}), described by a matrix with variables $x_1^{\pm 1}, \ldots, x_d^{\pm 1}$ and denoted here by $L_G$. We call its determinant the \emph{Laplacian (determinant) polynomial} $D_G = D_G(x_1, \ldots, x_d)$. (For finite graphs the term Laplacian polynomial is often used for the characteristic polynomial of the integer Laplacian matrix.) When $D_G \ne 0$, we use the main result of \cite{LSW90} to characterize the Mahler measure of $D_G$ as the complexity growth rate of the finite quotient graphs lying between $G$ and $\o G$. When all edge weights of $G$ are $1$, we recover a consequence of a more general result of Lyons \cite{Ly05} (see also \cite{BL12}). We show that Lehmer's question is equivalent to a question about graph complexity growth rates of 1-periodic graphs with edge weights equal to $\pm 1$.  

A 1-periodic \emph{plane graph} $G$ (that is, a graph embedded in the plane) determines an infinite link by the medial construction.  Its quotient by the $\Z$-action can be regarded as a (finite) link in an unknotted thickened annulus. The Alexander polynomial of the complement determines the Laplacian polynomial of $G$. We follow with some speculations about Lehmer's question and links.  

In the last section we present useful results about unweighted $d$-periodic graphs. They will not surprise some experts. However, as far as we know they do not appear in previous literature. 

  \bigskip

\ni {\bf Acknowledgements.} It is the authors' pleasure to thank Abhijit Champanerkar, Eriko Hironaka, Matilde Lalin and Chris Smythe for helpful comments and suggestions. 

\section{Complexity of finite graphs.}

%
%
%
%
%
%

Consider a finite graph $G$ with vertex set $V(G)=\{v_1, \ldots, v_n\}$ and edge set $E(G)=\{e_1, \ldots, e_m\}$. The graph  is allowed to have multiple edges. Loops will not affect affect results here and can be ignored. We assume also that the edges $e \in E(G)$ have \emph{weights} $w_e \in \Z$. (Generally $w_e$ will be $1$ or $-1$.) The graph $G$ is \emph{unweighted} if every weight is $1$.

The \emph{adjacency matrix} of $G$ is the $n \times n$ matrix 
$A = (a_{i, j})$ such that $a_{i, j}$ is the sum of the weights of edges between $v_i$ and $v_j$, for $i \ne j$. Diagonal entries of $A$ are zero. Define $\d= (\d_{i, j})$ to be the $n \times n$ diagonal matrix with $\d_{i,i}= \sum_j a_{i,j}.$

\begin{definition} \label{complexity} The \emph{Laplacian matrix} $L_G$ of a finite graph $G$ is $\d - A$. The abelian group presented by $L_G$ is the \emph{Laplacian group} of $G$, denoted by ${\L}_G$. The \emph{(torsion) complexity} $\k_G$ is the order of the torsion subgroup $T{\L}_G$
\end{definition}

When $G$ is connected and unweighted, the nullity of the Laplacian matrix $L_G$ is equal to 1 (see \cite{GR01}), and hence the Laplacian group ${\L}_G$ decomposes as the direct sum of $\Z$ and the torsion subgroup $T{\L}_G$. In this case, the Matrix Tree Theorem \cite{Tu84} implies that $\k_G$ is equal to the number of spanning trees of $G$.

More generally we  define \emph{tree complexity}  $\t_G$ of a connected graph $G$ by
\begin{equation*} \t_G =\bigg \vert \sum_{T} \prod_{e \in E(T)} w_e\bigg \vert, \end{equation*}
where the summation is taken over all spanning trees of $G$. If $G$ not connected, then we define $\t_G$ to be the product of the tree complexities of its connected components. Again by \cite{Tu84}, we have $\t_G = \k_G$ if and only if $\t_G$ is nonzero; for connected $G$ this common value is equal to any $(n-1) \times (n-1)$ principal minor of $L_G$. However,  the following example shows that $\t_G$ can vanish while $\k_G$ is positive.

\begin{figure}
\begin{center}
\includegraphics[height=2 in]{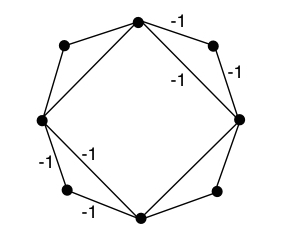}
\caption{Graph $G$ with $\t_G=0$ and $\k_G = 9$}
\label{milnor}
\end{center}
\end{figure}

\begin{example}  Consider the connected graph $G$ in Figure \ref{milnor}. Unlabeled edges here and throughout will be assumed to have weight $1$. The Laplacian matrix $L_G$ is square of size $8$. (See Example \ref{circ} below for a quick way to find $L_G$.) A routine calculation shows that any principal $7 \times 7$ minor of $L_G$ vanishes, and hence $\t_G=0$.   However, the absolute value of the greatest common divisors of the $6 \times 6$ minors of $L_G$ is 9. Hence $\k_G =9$. 
\end{example}

%
%
%
%
%
%
%


%
%

\section{$d$-Periodic graphs.} We regard $\Z^d$ as the multiplicative abelian group freely generated by $x_1, \ldots, x_d$. We denote the Laurent polynomial ring $\Z[\Z^d] = \Z[x_1^{\pm 1}, \ldots, x_d^{\pm 1}]$ by $\Rd$. As an abelian group $\Rd$ is generated freely by monomials $x^\s = x_1^{s_1} \ldots x_d^{s_d}$, where $\s = (s_1, \cdots, s_d) \in \Z^d$. 

Let $G$ be a graph that is  \emph{$d$-periodic}. By this we mean that $G$ has a cofinite free $\Z^d$-action by automorphisms that preserves edge weights. (By \emph{cofinite} we mean the quotient  graph $\o G$ is finite. The action is \emph{free} if the stabilizer of any edge or vertex is trivial.) Such a graph $G$ is necessarily locally finite. The vertex set $V(\o G)$ and the edge set $E( \o G)$ consist of finitely many vertex orbits $\{v_{1, \s}\mid \s \in \Z^d\}, \ldots, \{v_{n,\s}\mid \s \in \Z^d\}$ and weighted-edge orbits $\{e_{1, \s}\mid \s \in \Z^d\}, \ldots, \{e_{m, \s}\mid \s \in \Z^d\}$, respectively. The $\Z^d$-action is determined by 
\begin{equation} x^{\s'} \cdot v_{i, \s} = v_{i, \s+ \s'}, \quad \quad   x^{\s'} \cdot e_{j, \s} = e_{j, \s+ \s'},\end{equation}
where $1\le i \le n,\ 1 \le j \le m$ and $\s, \s' \in \Z^d$. (When $G$ is embedded in some Euclidean space with $\Z^d$ acting by translation, it is usually called a \emph{lattice graph}. Such graphs arise frequently in physics, for example in studying crystal structures.)

When $d>1$ we can think of $G$ as covering a graph $\o G$ in the $d$-torus $\T^d = \R^d/\Z^d$.  When $d=1$, $G$ covers a graph $\o G$ in the annulus $\S^1 \times I$. In either case the cardinality  $|V(\o G)|$ is equal to the number $n$ of vertex orbits of $G$, while $|E(\o G)|$ is the number $m$ of edge orbits. The projection map is given by $v_{i, \s} \mapsto v_i$ and $e_{j, \s} \mapsto e_j$.

If $\La \subset \Z^d$ is a subgroup, then the intermediate covering graph in $\Bbb R^d/\La$ will be denoted by $G_\La$.  The subgroups $\La$ that we will consider have index $r < \infty$, and hence $G_\La$ will be a finite $r$-sheeted cover of $\o G$ in the $d$-dimensional torus  $\Bbb R^d/\La$.

Given a $d$-periodic graph $G$, the Laplacian matrix is defined to be the $n \times n$-matrix $L_G = \d - A$, where now $A = (a_{i,j})$ is the \emph{weighted adjacency matrix} with each non-diagonal entry $a_{i, j}$  equal to the sum of  monomials $c_e x^\s$ for each edge $e \in E( \o G)$ between $v_{i, \bf{0}}$ and $v_{j, \s}$. (Again, ignoring loops, each diagonal entry of $A$ is zero.) The matrix $\d=(\d_{i,j})$ is the same diagonal matrix that we associate to $\o G$. 

The matrix $L_G$ presents a finitely generated $\Rd$-module, the \emph{Laplacian module} of $G$.
The \emph{Laplacian (determinant) polynomial} $\De_G$ is the determinant of $L_G$. Examples can be found in \cite{LSW14}.

The following is a consequence of the main theorem of \cite{Fo93}. It is made explicit in Theorem 5.2 of \cite{Ke11}. 

\begin{prop}\label{poly} \cite{Ke11} Let $G$ a $d$-periodic graph. Its Laplacian polynomial has the form \begin{equation}\label{polys} \De_G = \sum_F\  \prod_{e \in E(F)} c_e \prod_{\rm Cycles\ of\ F} (2-w -w^{-1}),\end{equation} where  the sum is over all cycle-rooted spanning forests $F$ of $\o G$, and $w, w^{-1}$ are the monodromies of the two orientations of the cycle. \end{prop} 

A \emph{cycle-rooted spanning forest} (CRSF) of $\o G$  is a subgraph of $G$  containing all of $V$ such that each connected component has exactly as many vertices as edges and therefore has a unique cycle. The element $w$ is the monodromy of the cycle, or equivalently, its homology in $H_1(\T^d; \Z) \cong \Z^d$.
See \cite{Ke11} for details. 

A $d$-periodic graph need not be connected. In fact, it can have countably many connected components. Nevertheless, the number of  $\Z^d$-orbits of components, henceforth called \emph{component orbits}, is necessarily finite. 

\begin{prop} \label{components}  If $G$ is a $d$-periodic graph with component orbits $G_1, \ldots, G_t$, then \break $\De_G = \De_{G_1}\cdots \De_{G_t}$. \end{prop}

\begin{proof} After suitable relabeling, the Laplacian matrix for $G$ is a  block diagonal matrix with diagonal blocks equal to the 
Laplacian matrices for $G_1, \ldots, G_s$. The result follows immediately. 
\end{proof}

\begin{prop} \label{zeropoly} Let $G$ a $d$-periodic graph. Its Laplacian polynomial $\De_G$ is identically zero if $G$ contains a closed component. The converse statement is true when $G$ is unweighted.
\end{prop} 

\begin{proof} If $G$ contains a closed component, then some component orbit $G_i$ consists of closed components.  We have $\De_{G_i}=0$ by \ref{poly}, since all cycles of $\overline{G_i}$ have monodromy 0. By Proposition \ref{components}, $\De_G$ is identically zero. 

Conversely, assume $G$ is unweighted and no component is closed. Each component of $\o G$ must contain a nontrivial cycle. We can extend this collection of cycles to a cycle rooted spanning forest $F$ with no additional cycles.  The corresponding summand  in \ref{poly} has positive constant coefficient. Since every summand has nonnegative constant coefficient,  $\De_G$ is not identically zero.

\end{proof}

%
%
%
%
%
%
%
%
%
%

\begin{definition} \label{mahler} The \emph{Mahler measure} of a nonzero polynomial 
$f(x_1, \ldots, x_d) \in \Rd$ is 
\begin{equation*} M(f) =\exp \int_0^1 \ldots \int_0^1 \log|f(e^{2\pi i \theta_1}, \ldots, e^{2\pi i \theta_d})| d\theta_1 \cdots d\theta_d. \end{equation*}

\end{definition} 

\begin{remark} (1)  The integral in Definition \ref{mahler} can be singular, but nevertheless it converges. 
(See \cite{EW99} for two different proofs.)  If $u_1, \ldots, u_d$ is another basis for $\Z^d$, then $f(u_1, \ldots, u_d)$ has the same logarithmic Mahler measure as $f(x_1, \ldots, x_d)$.  \smallskip

(2) When $d=1$, Jensen's formula shows that $M(f)$ can be described in a simple way. If $f(x) = c_s x^s+ \cdots c_1 x + c_0$, $c_0c_s \ne 0$, then
\begin{equation*} M(f) = |c_s| \prod_{i=1}^s \max\{\log |\lambda_i|, 1\},\end{equation*}
where $\lambda_1, \ldots, \lambda_s$ are the roots of $f$. \smallskip

(3) If $f, g \in \Rd$, then $M(fg) = M(f)M(g)$. Moreover, $M(f) =1$ if and only if $f$ is a unit or a unit times a product of 1-variable cyclotomic polynomials, each evaluated at a monomial of $\Rd$
(see \cite{Sc95}).

\end{remark}

\begin{theorem}(cf. \cite{Ly05}) \label{limit} If $G$ is a $d$-periodic graph with nonzero Laplacian polynomial $\De_G$,  then  
\begin{equation}  \limsup_{\langle \La \rangle \to \infty} \frac{1}{|\Z^d/\La|} \log \k_{G_\La}=
\log M(\De_G), \end{equation}
where $\La$ ranges over all finite-index subgroups of $\Z^d$, and $\langle \La \rangle$  denotes the minimum length of a nonzero vector in $\La$. When $d=1$,  the limit superior can be replaced by an ordinary limit. 
\end{theorem}

\begin{remark} \label{remarks} (1) The condition $\langle \La \rangle \to \infty$ ensures that fundamental region of $\La$ grows in all directions. 

(2) If $G$ is unweighted,  $\k_{G_\La}= \t_{G_\La}$ for every $\La$. In this case, Theorem \ref{limit} is proven in \cite{Ly05} with the limit superior replaced by ordinary limit. 

(3) When $d=1$, the finite-index subgroups $\La$ are simply $\Z/r\Z$, for $r >0$. In this case, we write $G_r$ instead of $G_\La$.

(4) When $d>1$,  a recent result of V. Dimitrov \cite{Di16} asserts that the limit superior in Theorem \ref{limit} is equal to the ordinary limit along sequences of sublattices $\La$ of the form $N\cdot \Z^d$, where $N$ is a positive integer. 

 \end{remark}

 \begin{figure}
\begin{center}
\includegraphics[height=3 in]{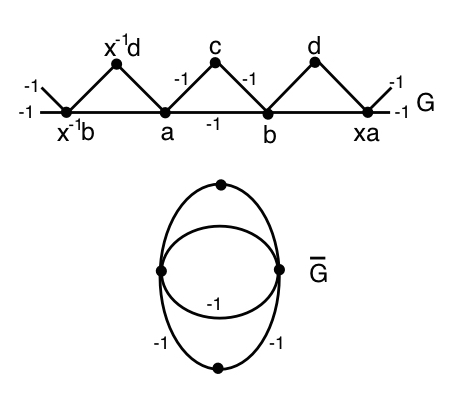}
\caption{1-Periodic graph $G$ with $\t_{G_r}=0$ for all $r \ge 1$}
\label{circulantdiagram}
\end{center}
\end{figure}

Before proving Theorem \ref{limit} we give an example that demonstrates the need for defining graph complexity as we do. 

\begin{example} \label{circ} Consider the 1-periodic graph $G$ in Figure \ref{circulantdiagram}. As before, unlabeled edges are assumed to have weight $1$. Generators for the Laplacian module are indicated. The Laplacian matrix is 
$$L_G= \begin{pmatrix} 0 &  1-x^{-1} & 1 & -x^{-1} \\ 1-x & 0 & 1 & -1 \\
1 & 1 & -2 & 0 \\ -x & -1 & 0 & 2 \end{pmatrix},$$ and 
$\De_G(x) = 9 (x-2+x^{-1}).$

The quotient  $G_2$ is the finite graph in Example \ref{milnor}. The  Laplacian matrix of any $G_r$ is easily described as a block matrix where $x$ is replaced by the companion (permutation) matrix for $x^r -1$. It is conjugate to a the diagonal block matrix ${\rm Diag}[L_G(1), \ldots, L_G(\zeta^{r-1})]$, where $\zeta$ is a primitive $r$th root of unity. The matrix $L_G(1)$ is the $4 \times 4$ Laplacian matrix of $\bar G$, 
$$ L_{\bar G}= \begin{pmatrix} 0 &  0 & 1 & -1\\ 0 & 0 & 1 & -1 \\
1 & 1 & -2 & 0 \\ -1 & -1 & 0 & 2 \end{pmatrix},$$ which has nullity 2. 
Hence the tree complexity $\t_{G_r}$ vanishes for every $r$. Nevertheless, by Theorem \ref{limit} the (torsion) complexity $\k_{G_r}$ is nontrivial and has exponential growth rate equal to $9$. One can verify directly that the Laplacian subgroup ${\cal L}_{G_r}$ is 
isomorphic to $\Z^2 \times (\Z/ 3^{r-1}\Z)^2.$\end{example} 

We proceed with the proof of Theorem \ref{limit}. 
 
 \begin{proof} The proof that we present is a direct application of a  theorem of D. Lind, K. Schmidt and T. Ward (see \cite{LSW90} or  Theorem 21.1 of \cite{Sc95}). We review the ideas for the reader's convenience. 
 
Recall that the Laplacian module ${\L}_G$ is the finitely generated module over the ring $\Rd$ with presentation matrix equal to the $n \times n$ Laplacian matrix $L_G$. We let  $\T$ be the additive circle group $\R/\Z$, and we consider the Pontryagin dual group $\widehat {\L}_G= {\rm Hom}({\L}_G, \T)$. We regard ${\L}_G$ as a discrete topological space. Endowed with the compact-open topology,
 $\widehat{\L}_G$ is a compact $0$-dimensional space. Moreover, the module actions of $x_1, \ldots, x_d$ determine commuting homeomorphisms $\si_1, \ldots, \si_d$ of $\widehat {\L}_G$. Explicitly, $(\si_j \rho)(a) = \rho(x_j a)$ for every $a \in {\L}_G$. Consequently, $\widehat {\G}_G$ has a $\Z^d$-action $\si: \Z^d \to \text{Aut}(\widehat {\L}_G)$.

The pair $(\widehat {\L}_G, \si)$ is an algebraic dynamical system, well defined up to topological conjugacy (that is, up to a homeomorphism of $\widehat {\L}_G$ respecting the $\Z^d$ action). In particular its periodic point structure is well defined.  

Topological entropy $h(\si)$ is a well-defined quantity associated to $(\widehat {\L}_G, \si)$, a measure of complexity of the $\Z^d$-action $\si$. We refer the reader to \cite{LSW90} or \cite{Sc95} for the definition.

For any subgroup $\La$ of $\Z^d$, a $\La$-periodic point is a member of $\widehat {\L}_G$ that is fixed by every element of $\La$. The set of $\La$-periodic points is a finitely generated abelian group isomorphic to the Pontryagin dual group ${\rm Hom}({\L}_G/\La {\L}_G), \T)$.

The group ${\L}_G/\La {\L}_G$ is the Laplacian module of the quotient graph $G_\La$. As a finitely generated abelian group, it decomposes 
as $\Z^{\b_\La} \oplus T({\L}_G/\La {\L}_G)$, where $\b_\La$ is the 
rank of ${\L}_G/\La {\L}_G$ and $T( \cdots )$ denotes the (finite) torsion subgroup. The Pontryagin dual group consists of $P_\La = |T({\L}_G/\La {\L}_G)|$ tori each of dimension $\b_\La$. By Theorem 21.1 of \cite{Sc95}, the topological entropy $h(\si)$ is:  \begin{equation*} h(\si) = \limsup_{\langle \La \rangle \to \infty}  \frac{1}{|\Z^d/\La|} \log P_\La =\limsup_{\langle \La \rangle \to \infty}  \frac{1}{|\Z^d/\La|} \log \k_\La.  \end{equation*}
Since the matrix $L_G$ that presents ${\L}_G$ is square, $h(\si)$ can be computed also as the logarithm of the Mahler measure $M(\det L_G)$ (see Example 18.7(1) of \cite{Sc95}).  The determinant of $L_G$ is, by definition, the Laplacian polynomial $\De_G$. Hence the proof is complete. 
\end{proof}

\section{Lehmer's question} In \cite{Le33} D.H. Lehmer 
asked the following yet unresolved question. 

\begin{question} \label{LQ} Do there exist  integral polynomials with Mahler measures arbitrarily close but not equal to 1? 

\end{question} 

Lehmer discovered the polynomial $x^{10} + x^9 - x^7 - x^6 - x^5 - x^4 - x^3 +x +1,$ which has Mahler measure equal to $1.17628...$, but he could do no better. Lehmer's question remains unanswered despite great effort including extensive computer-aided searches \cite{Bo80, Bo81, MS12, Mo98, Ra94}.

Topological and geometric perspectives of Lehmer's question have been found \cite{Hi03}.  In \cite{SW07} we showed that Lehmer's question is equivalent to a question about Alexander polynomials of fibered hyperbolic knots in the lens spaces $L(n, 1), n>0$. (Lens spaces arose from the need to consider polynomials $f(x)$ with $f(1) =n \ne 1$.) 
Here we present another, more elementary equivalence, in terms of graph complexity.

We will say that a Laurent polynomial $f(x) \in {\cal R}_1$ is \emph{palindromic} if 
$f(x^{-1}) = f(x)$. 
By a theorem of C. Smyth \cite{Sm71} it suffices to restrict our attention to palindromic polynomials when attempting to answer Lehmer's question.  

\begin{prop} \label{realized} A polynomial $\De(x)$ is the Laplacian polynomial of a 1-periodic graph if and only if it has the form $(x-2+x^{-1})^2 f(x)$, where $f(x)$ is a palindromic polynomial. 
\end{prop}

\begin{proof} The Laplacian polynomial $\De(x)$ of any $1$-periodic graph is palindromic. This is easy to see from the symmetry of the matrix $L_G$. Since the row-sums of $L_G$ become zero when $x=1$, it follows also that $x-1$ divides $D(x)$. (Both observations  follow also from Proposition \ref{poly}.) 
Palindromicity requires that the multiplicity of $x-1$ be even. Hence $\De(x)$ has the form $(x-2+x^{-1})^2 f(x)$, where $f(x)$ is palindromic. 

In order to see the converse assertion, consider any polynomial of the form $p(x)=(x-2-x^{-1})^2f(x)$, where $f(x)$ is palindromic. A straightforward induction on the degree of $f(x)$ shows that we can pair each term $\pm x^s$ with $\pm x^{-s}$, and then write $p(x)$
as a sum of terms $\pm(x^s-2+x^{-s})$. Then $p(x)$ is the Laplacian polynomial of a 1-periodic graph, constructed as in the following example. 
\end{proof} 

\begin{example} \label{circulant} Multiplying Lehmer's polynomial $f(x)= x^{10} + x^9 - x^7 - x^6 - x^5 - x^4 - x^3 +x +1$ by the unit $x^{-5}$ and then by $x-2+x^{-1}$ yields
$x^6-x^5-x^4+x^2+x^{-2}-x^{-4}-x^{-5}+x^{-6},$
which in turn can be written as 
$$(x^2-2+x^{-2}) -(x^4-2+x^{-4})-(x^5-2+x^{-5})+ (x^6 -2 + x^{-6}).$$
This the Laplacian polynomial of a 1-periodic graph $G$. The quotient graph $\o G$ is easily described. It has a single vertex, two edges with weight $+1$ and two with $-1$. The $(+1)$-weighted edges wind twice and six times, respectively,  around the annulus in the direction corresponding to $x$. The $(-1)$-weighted edges wind four and five times, respectively, in the opposite direction.
\end{example}

\begin{theorem} \label{LQ} Lehmer's question is equivalent to the following. 
Given $\e >0$, does there exist a 1-periodic graph $G$ such that 
$$1 < \lim_{r \to \infty}  (\t_{G_r})^{1/r} < 1+\e?$$ 
\end{theorem}

\begin{proof} When investigating Lehmer's question it suffices to consider polynomials of the form $(x-2+x^{-1})f(x)$, where $f(x)$ is  palindromic and irreducible. By Proposition \ref{realized} any such polynomial is realized as the Laplacian polynomial of a 1-periodic graph $G$ with a single vertex orbit. As in Example \ref{circ} the Laplacian matrix $L_{G_r}$ of any finite quotient $G_r$ can be obtained from $(x-2+x^{-1})f(x)$ by substituting for $x$ the companion matrix for $x^r-1$. Hence the nullity of $L_{G_r}$ is 1 provided that $f(x)$ is not a cyclotomic polynomial (multiplied by a unit), a condition that we can assume without loss of generality. 
Hence $\kappa_{G_r} = \tau_{G_r}$ for each $r$ (see discussion following Definition \ref{complexity}.) Theorem \ref{limit} completes the proof. 

\end{proof}

\begin{remark}

(1) The conclusion of Theorem \ref{LQ} does not hold if we restrict ourselves to unweighted graphs. By Theorem \ref{absolute} below,  the Mahler measure of the Laplacian polynomial of any 1-periodic graph with all edge weights equal to 1 is at least 2. \bigskip

(2) If a 1-periodic graph $G$ as in Example \ref{circulant} can be found with $M(D_G)$ less than Lehmer's value $1.17628...$, then by results of  \cite{Mo08} some edge of $\o G$ must wind around the annulus at least 29 times. \end{remark}

The cyclic 5-fold cover of the graph in Example \ref{circulant} contains the complete graph on 5 vertices, and hence it is nonplanar.  Hence we ask: 

\begin{question} \label{question1} Is Theorem \ref{LQ} still true if we require that the graphs $G$ be planar?  
\end{question}

\section{Laplacian polynomials of plane 1-periodic graphs} A finite plane graph determines a diagram of a \emph{medial link} by a simple procedure in which each edge of the graph is replaced by two arcs as in Figure \ref{convention1}. 
If the graph is unweighted then the resulting diagram is alternating as in Example  \ref{gridlinks} below. (The reader is invited to sketch the medial link associated to Figure \ref{milnor}. It is a non-alternating boundary link. According to \cite{Co70} it was introduced by J. Milnor and is the first link known to have zero Alexander polynomial.) 

It is well known that the Laplacian matrix of a finite plane graph $G$ is an unreduced Goeritz matrix of the associated link \cite{Go33}. The reduced matrix, obtained by deleting a row and column, presents the first homology group of the 2-fold cyclic cover of $\S^3$ branched over the link. (See \cite{Re74, Tr85, Li97} for additional details.)
 
When $G$ is a 1- or 2-periodic graph we apply the same construction to produce a diagram $\D$ of an infinite link $\ell$. It has a finite quotient diagram $\o \D$ modulo the induced $\Z$- or $\Z^2$-action on $G$. 

For $d=1$ we regard $\o \D$ in an annulus where it describes a link $\o \ell = \o \ell_1 \cup \cdots  \cup \o \ell_\mu$ in a solid unknotted torus $V$. The complement $({\rm int} V) \setminus \o \ell$ is homeomorphic to the $\S^3 \setminus \hat \ell$, where  $\hat \ell$ is the link $\o \ell \cup C$ formed by the union of $\o \ell$ with a meridian $C$ of $V$. The meridian acquires an orientation induced by the infinite cyclic action on $\D$.  The following result relates the Laplacian polynomial $\De_G$ to the Alexander polynomial $\Delta_{\hat \ell}$.

\begin{figure}
\begin{center}
\includegraphics[height=1  in]{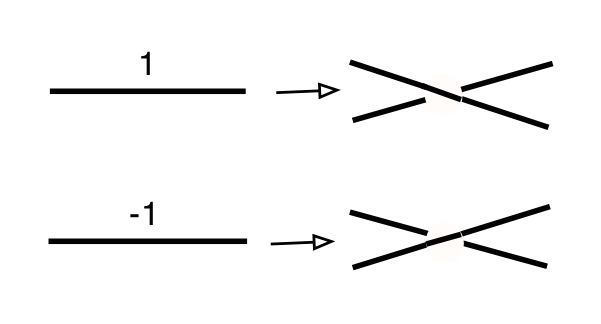}
\caption{Medial link construction for $(\pm 1)$-edge weighted graphs}
\label{convention1}
\end{center}
\end{figure}

\begin {theorem} \label{alex} Let $G$ be a plane 1-periodic graph and $\hat \ell$ the encircled link $\o \ell \cup C$. Then 
$$\De_G(x) \  {\buildrel \cdot \over =}\   (x-1)\ \Delta_{\hat \ell}(-1, \ldots, -1, x),$$
where ${\buildrel \cdot \over =}$  indicates equality up to multiplication by units in $\Z[x^{\pm 1}]$. 

\end{theorem} 

\begin{proof} We abbreviate $\Delta_{\hat \ell}(-1, \ldots, -1, x)$ by $F(x)$ and compute it via Fox calculus. Such a calculation can be done using the link group 
$\pi_1(\R^3 \setminus \hat \ell)$ and augmentation homomorphism $\phi: \Z[ \pi_1(\R^3 \setminus \hat \ell) ] \to \Z[x^{\pm 1}]$ that maps a meridian of $C$ to $x$ and meridians of $\o \ell$ to $-1$. Instead we will make use of $\pi_1(\R^3 \setminus \ell)$, an infinite-index subgroup of $\pi_1(\R^3\setminus \hat \ell)$. It has a countable Dehn presentation associated to the diagram $\D$. The presentation has generators $a_\nu, b_\nu, \ldots, \nu \in \Z$ corresponding to bounded regions together with a generator $u$ corresponding to one of the two unbounded regions of $\D$. (The second unbounded region, called the \emph{base region}, is labeled zero.) Here $a_\nu, b_\nu, \ldots$ is short-hand for the families $\{ x^\nu a x^{-\nu}\}, \{ x^\nu b x^{-\nu}\}, \ldots$, indexed by $\nu \in \Z$. The generators $a, b, \ldots$ are arbitrary but fixed orbit representatives. We regard $u$ as $\{u_\nu \mid u_\nu = u_{\nu +1}\}$. 

We recall that Dehn generators are related to the more-familiar Wirtinger generators once an orientation of the link $\ell$ is given. A Dehn generator corresponds to a loop that begins at a base-point above the plane, pierces the plane in the region of the Dehn generator, and returns through the base region. While the Dehn presentation requires no orientation of $\ell$, the restriction of the augmentation homomorphism $\phi$ to $\Z[\pi_1(\R^3 \setminus \ell)]$ is determined only with coefficients modulo 2. For this, we checkerboard color the diagram $\D$ with black and white in such a way that black regions contain the vertices of $G$. Under the augmentation homomorphism \emph{white generators} $a_\nu = x^\nu a x^{-\nu}$, those corresponding to white regions, map to $x^\nu$. \emph{Black generators},  those belonging to black regions, map to $-x^\nu$. 

Relations of the Dehn presentation correspond to crossings of $\D$, as in Figure \ref{convention}. The relation arising from the crossing at the top of the diagram is 
$a b^{-1} =c d^{-1}$ while the crossing at the bottom we may write as   $bd^{-1}= ac^{-1}$.  Like the generators, the relations come in countable families. Applying Fox calculus to relations, we obtain $a_\nu + b_\nu = c_\nu + d_\nu$ (top) and $b_\nu + d_\nu= a_\nu + c_\nu $ (bottom). We denote the collection of  Dehn relations by ${\cal R}$. 

\begin{figure}
\begin{center}
\includegraphics[height=2.5  in]{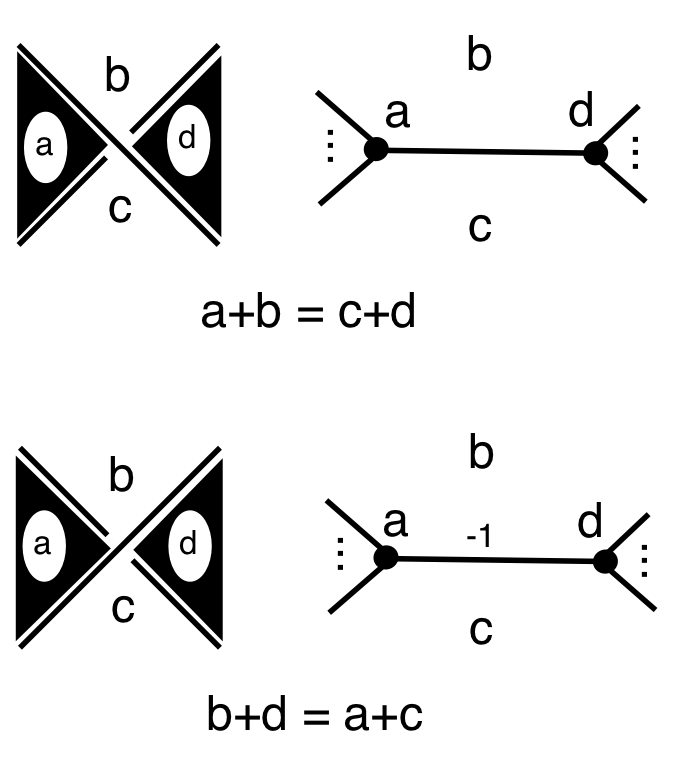}
\caption{Dehn relations}
\label{convention}
\end{center}
\end{figure}

Fix a black region of $\D$ (vertex of $G$) and identify it with the generator $a$ in Figure \ref{convention}. When we can add the relations involving $a$, the white generators cancel in pairs and we are left with the Laplacian relation of $G$ at the vertex. Note that this collection ${\cal R'}$ of Laplacian relations is a consequence of ${\cal R}$. 

Consider the dual graph $G^*$ and choose a spanning tree $T^*$ rooted at the base region of $\D$. Each edge of $T^*$ crosses a unique edge of $G$, and hence $T^*$ corresponds to a subset  ${\cal R}_0$ of ${\cal R}$.  Using these relations we can express each white generator in terms of black generators. Lemma \ref{lem} below will show that the remaining relations in ${\cal R}$ are consequences of ${\cal R}_0 \cup {\cal R'}$. Once we use the relations ${\cal R}_0$ to eliminate white generators and rewrite the remaining relations, we see the presentation of ${\cal L}_G$ given by the Laplacian matrix $L_G$. Since the polynomial $F(x)$ 
is the $0$th determinantal invariant of the module divided by $x-1$, the proof is complete.

\end{proof} 
 
\begin{lemma} \label{lem} Let $G$ be a plane 1-periodic graph, and let ${\cal R, R'}$  be the Dehn and Laplacian relations of the link diagram resulting as above from the medial construction. Let ${\cal R}_0$ be the Dehn relations corresponding to the edges of a spanning tree for the dual graph $G^*$ rooted at the base region of $\D$. Then any Dehn relation in ${\cal R} \setminus {\cal R}_0$  is a consequence of ${\cal R}_0 \cup {\cal R'}$.\end{lemma} 

\begin{proof} Any Dehn relation in ${\cal R} \setminus {\cal R}_0$ corresponds to an edge of $G^*$ not contained in $T^*$. Adding the edge to $T^*$ results in a unique cycle $\g$. We proceed by induction on the number of vertices of $G$ enclosed by $\g$. 

If $\g$ encloses a single vertex, then, as in the proof of Theorem \ref{alex}, the relations corresponding to the edges of $\g$ add together to give the Laplacian relation at the vertex.

If $\g$ encloses more than one vertex, then it can be decomposed as 
a union of cycles $\g_1$ and $\g_2$ with edges of $G$ in common such that each cycle encloses fewer vertices than $\g$. The sum of the Dehn relations corresponding to the edges of $\g$ is equal to the addition of the sums coming from $\g_1$ and $\g_2$, the relations  corresponding to common edges canceling in pairs. By the induction hypothesis the later is a consequence of ${\cal R}_0 \cup {\cal R'}$.

\end{proof}

Question \ref{question1} might be approached by reversing the process of transforming a plane 1-periodic graph $G$ to an encircled link $\hat \ell$. Given any link $\hat \ell$ with an unknotted component $C$,  the Mahler measure of $\Delta_{\hat \ell}(x, -1, \ldots, -1)$ is the Mahler measure of some 1-periodic plane graph. (Neither the extra factor of $x-1$ nor the orientation of $C$ will affect the Mahler measure.)  However, the classification of Alexander polynomials of links with an unknotted component is a difficult problem that has been only partly solved \cite{Ki79, Co82, Tr82}

\section{Unweighted $d$-periodic graphs} We present some results about unweighted $d$-periodic graphs. They will not surprise some experts, but, as far as we know, they have not appeared elsewhere. In particular, we will show that our restatement of Lehmer's question in Theorem \ref{LQ} is not valid if we restrict ourselves to unweighted graphs. \bigskip

{\sl Graphs considered in this section are unweighted. }\bigskip

We call the limit in Theorem \ref{limit} the {\it complexity growth rate} of $G$, and denote it by $\g_G$.  Its relationship to the {\it thermodynamic limit} or {\it bulk limit} defined for a wide class of unweighted lattice graphs is discussed in  \cite{LSW14}. 

Denote by $R= R(\La)$ a fundamental domain of $\La$. Let $G\vert_R$ be the full unweighted subgraph of $G$ on vertices $v_{i, \s},\ \s \in R$. We denote by $\ell_R$ the corresponding medial link. 

 If  $G\vert_R$ is connected for each $R$, then $\{\tau_{G_\La}\}$ and $\{\tau_{G\vert_R}\}$ have the same exponential growth rates. (See Theorem 7.10 of \cite{LSW14} for a short, elementary proof. A more general result is Corollary 3.8 of \cite{Ly05}.) The bulk limit is defined by $\g_G/|V(\o G)|$. 

\begin{example} The \emph{$d$-dimensional grid graph} $\Gr_d$ is the unweighted graph with  vertex set $\Z^d$ and edges 
from $(s_1, \ldots, s_d)$ to $(s_1', \ldots, s_d')$ if $|s_i -s_i'|=1$ and $s_j = s_j',\ j \ne i$, for every $1 \le i \le d$.  Its Laplacian polynomial is 
\begin{equation*} \De(\Gr_d) = 2d -x_1-x_1^{-1}- \cdots - x_d-x_d^{-1}. \end{equation*}
When $d=2$, it is a plane graph. The graphs links $\ell_R$ are indicated in Figure \ref{gridlinks} for $\La = \langle x_1^2, x_2^2 \rangle$ on left and $\La= \langle x_1^3, x_2^3 \rangle$ on right.

\begin{figure}
\begin{center}
\includegraphics[height=2 in]{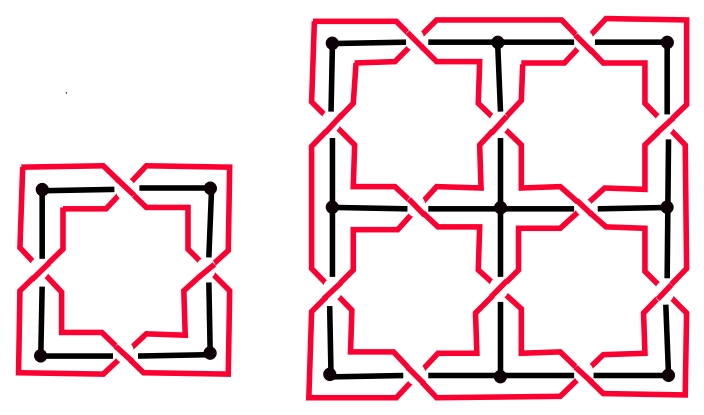}
\caption{Graphs $(\Gr_2)_R$ and associated links, $\La = \langle x_1^2, x_2^2 \rangle$ and $\langle x_1^3, x_2^3 \rangle$}
\label{gridlinks}
\end{center}
\end{figure}

\end{example}

The {\it determinant} of a link $\ell$, denoted here by ${\rm det}(\ell)$, is the absolute value of its 1-variable Alexander polynomial evaluated at $-1$.   
It follows from the Mayberry-Mott theorem \cite{BM54} that if $\ell$ is an alternating link that arises by the medial construction from a finite plane graph, edge weights $\pm 1$ allowed, then ${\rm det}(\ell)$ is equal to the tree complexity of the graph (see Appendix A.4 in \cite{BZ85}).
The following corollary is an immediate consequence of Theorem \ref{limit}.  It has been proven independently by Champanerkar and Kofman \cite{CK16}.

\begin{cor}\label{detgrowth} Let $G$ be a connected $d$-periodic unweighted plane graph, $d =1$ or $2$. Then $$\lim_{\langle \La \rangle \to \infty}  \frac{1}{\vert \Z^d/\La\vert} \log {\rm det}(\ell_R) = 
\g_{\De_G}.$$ \end{cor}

\begin{remark} (1) In \cite{CKP15} the authors consider as well more general sequences of links.  When $G = \Gr_2$, their results imply that:

$$\lim_{\langle \La \rangle \to \infty} \frac{2 \pi}{c(\ell_R)} \log {\rm det}(\ell_R)=  v_{oct},$$
where $c(\ell_R)$ is the number of crossings of $\ell_R$ and $v_{oct} \approx 3.66386$ is the volume of the regular ideal octohedron.\bigskip

(2) If Question \ref{question1} has an affirmative answer then Lehmer's question becomes a question about link determinants. 

\end{remark}

Grid graphs are the simplest unweighted $d$-periodic graphs,  as the 
following theorem shows. 

\begin{theorem} \label{min} Assume that $G$ is an unweighted connected $d$-periodic graph.  Then $\g_G\ge \g_{\Gr_d}$. 
\end{theorem} 

\begin{remark} The conclusion of Theorem \ref{min} does not hold without the hypothesis that $G$ is connected. Consider the 2-periodic graph $G$ consisting of 
countably many copies of $\Gr_1$ obtained from $\Gr_2$ by removing all vertical edges. Then $\g_G =0$ while $\g_{\Gr_2} > 0$. \end{remark}

The following lemma, needed for the proof of Corollary \ref{absolute}, is of independent interest. 

\begin{lemma} \label{nonincreasing} The sequence of complexity growth rates $\g_{\De_{\Gr_d}}$ is nondecreasing. 
\end{lemma} 

Doubling each edge of $\Gr_1$ results in a graph with Laplacian polynomial $2(x-2+x^{-1})$, which has Mahler measure $2 M (x-2+x^{-1}) = 2$. The following corollary states that this is minimum nonzero complexity growth rate. 

\begin{cor} \label{absolute} (Complexity Growth Rate Gap) Let $G$ be any unweighted $d$-periodic graph with Laplacian polynomial $\De_G$. If $\g_G \ne 0$, then 
$$\g_G \ge \log 2.$$
\end{cor}

Although $\De_{\Gr_d}$ is relatively simple, the task of computing its Mahler measure is not. It is well known and not difficult to see that $\g_{\Gr_d} \le \log 2d$. We will use a theorem of N. Alon \cite{Al90} to show that $\g_{\Gr_d}$ approaches $\log 2d$ asymptotically. 

\begin{theorem} \label{gridlim} (1) $\g_{\Gr_d} \le \log 2d$, for all $d \ge 1$.\\ 
(2) $\lim_{d \to \infty} \g_{\Gr_d}- \log 2d = 0.$ \end{theorem}

Asymptotic results about the Mahler measure of certain families of polynomials have been obtained elsewhere. However, the graph theoretic methods that we employ to prove Theorem \ref{min} are different from techniques used previously. \bigskip

Now suppose $H$ is a subgraph of $G$ consisting of one or more connected components of $G$, such that the orbit of $H$ under $\Z^d$ is all of $G$.  Let 
$\Gamma < \Z^d$ be the stabilizer of $H$.  Then $\Gamma\cong \Z^{d'}$ for some $d'\le d$, and its action on $H$ can be regarded as a cofinite free action of $\Z^{d'}$.   Consider the limit $$\g_H=\lim_{\<\La\>\to\infty} \frac{1}{|\Gamma/\La|} \log \k_{H_\La}$$
where $\La$ ranges over finite-index subgroups of $\Gamma$.

\begin{lemma}\label{new} Under the above conditions we have $\g_G=\g_H$.
\end{lemma}

\begin{proof}
Let $\La$ be any finite-index subgroup of $\Z^d$.  Then $H$ is invariant under $\La\cap\G$. The image of $H$ in the quotient graph $G_\La$ is isomorphic to  $H_{\La\cap\G}$.  

Note that the quotient $\overline H$ of $H$ by the action of $\G$ is isomorphic to $\overline G$, since the $\Z^d$ orbit of $H$ is all of $G$.  Since $G_\La$ is a $|\Z^d /\La|$-fold cover of $\overline G$ and $H_{\La\cap\G}$ is a $|\G/(\La\cap\G)|$-fold cover of $\overline H$, $G_\La$ comprises $k=|\Z^d /\La| / |\G/(\La\cap\G)|$ mutually disjoint translates of a graph that is isomorphic to $H_{\La\cap\G}$. Hence $\k_{G_\La}=\k_{H_{\La\cap\G}}^k$ and
$$\frac{1}{|\Z^d /\La|} \log \k_{G_\La}= \frac{1} {|\G/(\La\cap\G)|} \log \k_{H_{\La\cap\G}}.$$
Since $\<\La\cap\G\>\to\infty$ as $\<\La\>\to\infty$, we have $\g_G=\g_H$.
\end{proof}

\noindent {\it Proof of Theorem \ref{min}.}  
Consider the case in which $G$ has a single vertex orbit. Then for some $u_1,\ldots, u_m\in \Z^d$, with $m \ge d$, the edge set $E(G)$ consists of edges from $v$  to $u_i \cdot v$ for each $v\in V$ and $i=1,\ldots,m$.  Since $G$ is connected, we can assume after relabeling that $u_1, \ldots, u_d$ generate a finite-index subgroup of $\Z^d$. Let $G'$ be the be the $\Z^d$-invariant subgraph of $G$ with edges from $v$  to $u_i \cdot v$ for each $v\in V$ and $i=1,\ldots,d$.  Then $G'$ is the orbit of a subgraph of $G$ that is isomorphic to $\Gr_d$, and so by Lemma \ref{new}, $\g(\Gr_d)=\g(G')\le\g(G)$.

We now consider a connected graph $G$ having vertex families $v_{1, \s}, \ldots, v_{n, \s}$, where $n >1$. Since $G$ is connected, there exists an edge $e$ joining $v_{1, {\bf 0}}$ to some $v_{2, \s}$. Contract the edge orbit $\Z^d \cdot e$ to obtain a new graph $G'$ having cofinite free $\Z^d$-symmetry and complexity growth rate no greater than that of $G$.
Repeat the procedure with the remaining vertex families so that only $v_{1, \s}$ remains. The proof in the previous case of a graph with a single vertex orbit now applies. \qed \\


\noindent {\it Proof of Lemma \ref{nonincreasing}.}  Consider the grid graph $\Gr_d$. Deleting all edges 
in parallel to the $d$th coordinate axis yields a subgraph $G$ consisting of countably many mutually disjoint translates of $\Gr_{d-1}$. By Lemma \ref{new},   $\g_{\Gr_{d-1}}=\g_G  \le \g_{\Gr_d}$. 
\qed \\

\noindent {\it Proof of Corollary \ref{absolute}.}  By Lemma \ref{new} it suffices to consider a connected $d$-periodic graph $G$ with $\g_G$ nonzero. Note that $\g_{\Gr_1} =0$ while $\g_{\Gr_2} \approx 1.165$ is greater than $\log 2$.  By Theorem \ref{min} and Lemma \ref{nonincreasing} we can  assume that $d=1$. 

If $G$ has an orbit of parallel edges, we see easily that $\g_G \ge \log 2$.  Otherwise, we proceed as in the proof of  Theorem \ref{min}, contracting edge orbits to reduce the number of vertex orbits  without increasing the complexity growth rate.  If at any step we obtain an orbit of parallel edges, we are done; otherwise we will obtain a graph $G'$ with a single vertex orbit and no loops. 
If $G'$ is isomorphic to $\Gr_1$ then $G$ must be a tree; but then $\g_G=0$, contrary to our hypothesis. So $G'$ must have at least two edge orbits.  
Deleting excess edges, we may suppose $G'$ has exactly two edge orbits. 

The Laplacian polynomial $\De_{G'}$ has the form $4-x^r-x^{-r}- x^s-x^{-s}$, for some positive integers $r, s$. Reordering the vertex set of $G'$, we can assume without loss of generality that $r=1$. The following calculation is based on an idea suggested to us by Matilde Lalin.

$$\log M(\De_{G'}) = \int_0^1 \log \vert 4- 2 \cos( 2 \pi \theta)-2 \cos(2 \pi s \theta) \vert \ d\theta$$

$$=\int_0^1 \log \vert  2(1-\cos  (2 \pi \theta)) + 2(1 - \cos(2 \pi s \theta)) \vert \ d\theta$$

$$=\int_0^1 \log  \bigg( 4\sin^2(\pi \theta) + 4\sin^2(\pi s \theta) \bigg) \ d\theta.$$

\ni Using the inequality $(u^2+v^2) \ge 2 u v$, for any nonnegative $u, v$, we have:

$$\log M(\De_{G'}) \ge \int_0^1 \log \bigg( 8 \vert \sin( \pi \theta)\vert\  \vert \sin( \pi s \theta) \vert \bigg)\ d\theta$$

$$= \log 8 + \int_0^1 \log \vert \sin(  \pi \theta) \vert \ d\theta + \int_0^1 \log \vert \sin( \pi s \theta) \vert \ d\theta$$

$$= \log 8 + \int_0^1 \log \sqrt{\frac{1-\cos(2 \pi \theta)}{2}}\ d \theta + \int_0^1 \log \sqrt{\frac{1-\cos(2  \pi s\theta)}{2}}\ d \theta$$

$$=\log 8 + \int_0^1 \frac{1}{2} \log \bigg(\frac{2 - 2 \cos(2 \pi \theta)}{4}\bigg)\ d \theta +   \int_0^1 \frac{1}{2} \log \bigg(\frac{2 - 2 \cos(2 \pi s \theta)}{4}\bigg)\ d \theta$$

$$= \log 8 + \frac{1}{2}m(2 - x -x^{-1}) -\frac{1}{2} \log 4 + \frac{1}{2}m(2 - x^s - x^{-s})- \frac{1}{2} \log 4$$

$$= 3\log 2 + 0 - \log 2 +0 -\log 2 = \log 2.$$
\qed\\

Our proof of Theorem \ref{gridlim} depends on the following result of Alon. 

\begin{theorem}\label{lowerbound}\cite{Al90}   If $G$ is a finite connected $\rho$-regular unweighted graph, then 
$$\t_G \ge [\rho(1-\epsilon(\rho))]^{|V(G)|},$$ 
where $\epsilon(\rho)$ is a nonnegative function with $\epsilon(\rho)\to\infty$ as $\rho\to\infty$.
\end{theorem}

\noindent{\it Proof of Theorem \ref{gridlim}.} (1) The integral representing the logarithm of the Mahler measure of $\De_{\Gr_d}$ can be written
$$\int_0^1 \cdots \int_0^1 \log \bigg\vert2d - \sum_{i=1}^d 2 \cos(2 \pi \theta_i)\bigg\vert d\theta_1\cdots d\theta_d$$
$$= \log 2d + \int_0^1 \cdots \int_0^1 \log \bigg\vert1+\sum_{i=1}^d \frac{ \cos( 2 \pi \theta_i)}{d}\bigg\vert d\theta_1\cdots d\theta_d$$
$$= \log 2d + \int_0^1 \cdots \int_0^1 -\sum_{k=1}^\infty \frac{(-1)^k}{k}\bigg( \frac{\sum_{i=1}^d \cos(2 \pi \theta_i)}{d} \bigg)^k d\theta_1\cdots d\theta_d.$$
By symmetry, odd powers of $k$ in the summation contribute zero to the integration. Hence 
$$\log(\De_{\Gr_d})= \log 2d - \int_0^1 \cdots \int_0^1 \sum_{k=1}^\infty \frac{1}{2k}\bigg( \frac{\sum_{i=1}^d \cos(2 \pi \theta_i)}{d} \bigg)^{2k} d\theta_1\cdots d\theta_d \le \log 2d.$$

(2) Let $\La$ be a finite-index subgroup of $\Z^d$. Consider the quotient graph $(\Gr_d)_\La$. The cardinality of its vertex set is $|\Z^d/\La|$. The main result of \cite{Al90}, cited above as Theorem \ref{lowerbound}, implies that 
$$\t_{(\Gr_d)_\La} = \bigg((2 d) (1 - \mu(d))\bigg)^{|\Z^d/\La|},$$
where $\mu$ is a nonnegative function such that $\lim_{d\to \infty} \mu(d) =0$.  
Hence 
$$\lim_{d \to \infty}  \bigg (\frac{1}{|\Z^d/\La|} \log \t_{(\Gr_d)_\La} - \log 2 d\bigg ) = \lim_{d\to \infty} \log (1-\mu(d)) =0.$$
Theorem \ref{limit} completes the proof.  
\qed \\

\begin{remark} One can evaluate $\log M(\De(\Gr_d))$ numerically and obtain an infinite series representing
$\g_{\Gr_d} - \log 2d$. However, showing rigorously that the sum of the series approaches zero as $d$ goes 
to infinity appears to be difficult. (See \cite{SW00}, p. 16 for a heuristic argument.) 
\end{remark}

\bigskip

\ni Department of Mathematics and Statistics,\\
\ni University of South Alabama\\ Mobile, AL 36688 USA\\
\ni Email: \\
\ni  silver@southalabama.edu\\
\ni swilliam@southalabama.edu
\end{document}